\documentclass[leqno,12pt]{article}
\pdfoutput=1 
\usepackage{latexsym}
\usepackage{amssymb}
\usepackage{amsthm}
\usepackage{amsmath}

\newtheorem{theorem}{Theorem}[section]
\newtheorem{lemma}[theorem]{Lemma}
\newtheorem{corollary}[theorem]{Corollary}
\newtheorem{proposition}[theorem]{Proposition}
\newtheorem{conjecture}[theorem]{Conjecture}

\theoremstyle{definition}
\newtheorem{definition}[theorem]{Definition}

\theoremstyle{remark}
\newtheorem{remark}[theorem]{Remark}

\numberwithin{equation}{section}

\DeclareMathOperator*{\argmax}{argmax}
\DeclareMathOperator*{\linf}{liminf}
\usepackage{latexsym}
\usepackage{amssymb}
\usepackage{amsthm}
\usepackage{amsmath}
\pdfoutput=1 
\newcommand{\nc}{\newcommand}
\nc{\nt}{\newtheorem} 
\nc{\ip}[2]{\mbox{$\langle #1,#2 \rangle$}} 
\nc{\pf}{\noindent{\bf Proof\ \ }}
\nc{\finpf}{\hfill{$\Box$}\linespace}
\nc{\linespace}{\vspace{\baselineskip} \noindent} 
\nc{\R}{{\bf R}}
\nc{\cl}{\mbox{\rm cl}\,} 
\nc{\rb}{\mbox{\rm rb}\,}
\nc{\ri}{\mbox{\rm ri}\,}
\nc{\epi}{\mbox{\rm epi}\,}
\nc{\dom}{\mbox{\rm dom}\,}

\def\tto{\;{\lower 1pt \hbox{$\rightarrow$}}\kern -12pt
           \hbox{\raise 2.8pt \hbox{$\rightarrow$}}\;}

\usepackage{enumerate}
\usepackage{graphicx}

\usepackage{url}
\begin{document}



\author{D. Drusvyatskiy\thanks{%
    School of Operations Research and Information Engineering,
    Cornell University,
    Ithaca, New York, USA;
    {\tt dd379@cornell.edu}.
    Work of Dmitriy Drusvyatskiy on this paper has been partially supported by the NDSEG grant from the Department of Defense.}
  \and
  A.S. Lewis\thanks{%
  School of Operations Research and Information Engineering,
  Cornell University,
  Ithaca, New York, USA;
  {\tt http://people.orie.cornell.edu/{\raise.17ex\hbox{$\scriptstyle\sim$}}aslewis/}.
  Work of A. S. Lewis on this paper has been supported in part by National Science Foundation Grant DMS-0806057.
}}

\title{Generic nondegeneracy in convex optimization}


\date{May 6, 2010.}


\maketitle

\begin{abstract} 
We show that minimizers of convex functions subject to almost all linear perturbations are nondegenerate. An analogous result holds more generally, for {\em lower-$\bf{C}^2$} functions.    
\end{abstract}

\section{Introduction}
In this work we study the nature of minimizers of ``typical'' convex functions. 
 We model this question by considering a fixed extended-real-valued convex function
$f$, and then studying properties of minimizers of the perturbed function $x \mapsto f_v(x)=f(x)-v^{T}x$ that hold for {\em almost all} values of the data vector $v \in \R^n$ (in the sense of Lebesgue measure).

Classical theory shows that, given a proper convex function $f$, the perturbed function $f_v$ typically has at most one minimizer.  To see this, note first that we may assume $f$ is closed, since any minimizer of $f$ also minimizes its closure.  Now we observe that the Fenchel conjugate $f^*$ is differentiable almost everywhere on the interior of its domain, by Rademacher's theorem (see for example \cite[Theorem 9.60]{VA}), so for almost all vectors $v$, the subdifferential $\partial f^*(v)$ is either single-valued or empty.  The result now follows, since this subdifferential coincides with the set $(\partial f)^{-1}(v)$, which is exactly the set of minimizers of $f_v$.

Our aim is to strengthen this classical result.  Minimizers $x$ of the perturbed function $f_v$ are characterized by the property that the vector zero lies in the subdifferential $\partial f_v(x)$.  We prove, for almost all vectors $v$, that the minimizer $x$ is not only unique, but also {\em nondegenerate\/}, by which we mean that zero lies in the relative interior of the subdifferential:  $0 \in \ri \partial f_v(x)$ (or equivalently, the positive span $\R_+ \partial f_v(x)$ is a subspace).  The proof, following an idea of \cite{pataki-tuncel}, uses a result in geometric measure theory due to Larman \cite{Larman}.

As an example, consider the standard linear programming problem 
\[
\max_{x\in\R^n} \Big\{v^{T}x: a_i^{T}x\leq b_i ~(i=1,2,\ldots,m)\Big\},
\]
for given vectors $a_i \in \R^n$ and scalars $b_i \in \R$.
We can restate this problem as minimizing the perturbed function $f_v$ corresponding to the original function $f$ that takes the value zero on the feasible region and $+\infty$ elsewhere. Consider an optimal solution 
$\bar{x}$ and the corresponding index set of active constraints, $I=\{i:a_i^{T}\bar{x}=b_i\}$. Then we have 
\begin{align*}
\partial f_v(\bar{x})&= -v+\{\displaystyle\sum\limits_{i\in I} \lambda_ia_i: \lambda_i\geq 0\},\\
\ri \partial f_v(\bar{x})&= -v+\{\displaystyle\sum\limits_{i\in I} \lambda_ia_i: \lambda_i > 0\}.
\end{align*}
Thus the minimizer $\bar{x}$ of $f_v$ is nondegenerate exactly when there exists a dual-feasible solution 
$\lambda\in\R^m$ satisfying strict complementary slackness.  We hence recover the well-known fact that, for almost all objective functions, if a linear program has an optimal solution, then that solution is unique and furthermore corresponds to a strictly-complementary-slack dual solution.

For convex functions, critical points (those at which zero is a subgradient) coincide with minimizers.  For nonconvex functions, we can more generally consider nondegeneracy of critical points.  It transpires that our result on typical nondegeneracy extends in particular to all {\em lower-$\bf{C}^2$} functions (those functions locally representable as sums of convex functions and quadratics).  However, in more general contexts the result may fail.  The classical generalization of the subdifferential of a convex function is the Clarke generalized gradient \cite{Clarke}, but \cite{counter} presents a locally Lipschitz function $f\colon\R\rightarrow\R$, whose Clarke generalized gradient $\partial_c f$ at any point $x\in\R$ is the interval 
$[-x,x]$.  In this case, the perturbed function $f_v$ has a degenerate critical point for {\em every} non-zero value of $v$.

\section{Preliminaries}\label{sec:pre}
\subsection{Variational Analysis}
We recall some standard notions from variational analysis (see for example \cite{VA}).
Consider the extended real line $\overline{\R}:=\R\cup\{-\infty\}\cup\{+\infty\}$. We say that an extended-real-valued function is {\em proper} if it is never $\{-\infty\}$ and is not always $\{+\infty\}$.  

For a function $f\colon\R^n\rightarrow\overline{\R}$, we define the {\em domain} of $f$ to be $$\mbox{\rm dom}\, f=\{x\in\R^n: f(x)<+\infty\},$$ and we define the {\em epigraph} of $f$ to be $$\mbox{\rm epi}\, f= \{(x,r)\in\R^n\times\R: r\geq f(x)\}.$$  A function is {\em convex} when its epigraph is convex, and {\em closed} when its epigraph is closed. 

\begin{definition}
{\rm Consider a set $S\subset\R^n$ and a point $\bar{x}\in S$. The {\em regular normal cone} to $S$ at $\bar x$, denoted 
$\hat N_S(\bar x)$, consists of all vectors $v \in \R^n$ such that $$\langle v,x-\bar{x} \rangle \leq o(|x-\bar{x}|) \textrm{ for }x\in S,$$ 
where we denote by $o(|x-\bar{x}|) \textrm{ for }x\in S$ a term with the property that $$\frac{o(|x-\bar{x}|)}{|x-\bar{x}|}\rightarrow 0$$ when $x\stackrel{S}{\rightarrow} \bar{x}$ with $x\neq\bar{x}$.
} 
\end{definition}

\begin{definition}
{\rm Consider a set $S\subset\R^n$ and a point $\bar{x}\in S$.  The {\em limiting normal cone} to $S$ at $\bar{x}$, denoted $N_S(\bar{x})$, consists of all vectors $v\in\R^n$ such that there are sequences $x_r\stackrel{S}{\rightarrow} \bar{x}$ and $v_r\rightarrow v$ with $v_r\in\hat{N}_S(x_r)$.}
\end{definition}

In the presence of convexity, normal cones have a much simpler form.
\begin{theorem}\cite[Theorem 6.9]{VA} For a convex set $S\subset\R^n$ and a point $\bar{x}\in S$, the regular and the limiting normal cones coincide, and consist of all vectors $v \in \R^n$ such that $$\langle v,x-\bar{x} \rangle \leq 0 \textrm{ for all }x\in S.$$ 
\end{theorem}

Normal cones allow us to study geometric objects. We now define subdifferentials, which allow us to analyze behavior of functions.
\begin{definition} 
{\rm Consider a function $f\colon\R^n\rightarrow\overline{\R}$ and a point $\bar{x}\in\R^n$ where $f$ is finite. The {\em regular} and the {\em limiting subdifferentials} of $f$ at $\bar{x}$, respectively, are defined by 
\begin{align*}
\hat{\partial}f(\bar{x})&~=~ \big\{v\in\R^n: (v,-1)\in \hat{N}_{\mbox{{\scriptsize {\rm epi}}}\, f}(\bar{x},f(\bar{x}))\big\},\\  
\partial f(\bar{x})&~=~ \big\{v\in\R^n: (v,-1)\in N_{\mbox{{\scriptsize {\rm epi}}}\, f}(\bar{x},f(\bar{x}))\big\}.
\end{align*}
}  
\end{definition}

If the function $f$ is convex, both subdifferentials reduce to the classical convex subdifferential,
\[
\big\{ v \in \R^n : \langle v,x - \bar x \rangle \le f(x) - f(\bar x)~ \mbox{for all}~x \in \R^n \big\}.
\]

\begin{remark} 
{\rm For $x\in\R^n$ where $f(x)$ is not finite, we follow the convention that $\hat{\partial}f(x)=\partial f(x)=\emptyset$.
The regular and the limiting subdifferentials are always closed sets, and the regular subdifferential is convex.}
\end{remark}

Subdifferentials play the role of generalized gradients in the following sense.
\begin{theorem}\cite[Exercise 8.8]{VA} \label{ex}
Consider a function $f\colon\R^n\to\overline{\R}$ and a point $\bar{x}\in\R^n$. If $f$ can be written as $f=g+h$, where $g$ is finite at $\bar{x}$ and $h$ is $\bf{C}^1$ smooth on a neighborhood of $\bar{x}$, then 
\begin{align*}
\partial f(\bar{x})&=\partial g(\bar{x})+\nabla h(\bar{x}),\\ 
\hat{\partial} f(\bar{x})&=\hat{\partial} g(\bar{x})+\nabla h(\bar{x}).
\end{align*}
\end{theorem}

\begin{theorem}\cite[Theorem 12.12, 12.17]{VA}\label{thm:minty}
Let $f\colon\R^n\to\overline{\R}$ be a proper, convex function. Then on the set where the set-valued mapping 
$(I+\partial f)^{-1}$ takes nonempty values, it is single-valued and Lipschitz continuous with constant $1$. 
\end{theorem}

\begin{remark}
{\rm Theorem~\ref{thm:minty} is a special case of the celebrated theorem of Minty. See \cite{minty} or \cite[Section 12.B]{VA} for more details.}
\end{remark}

We now define a large and robust class of functions that includes both smooth functions and finite convex functions. 

\begin{definition}\cite[Theorem 10.33]{VA}\label{thm:char}
{\rm
A function $f\colon O\to{\R}$, where $O$ is an open set in $\R^n$, is said to be {\em lower-$\bf{C}^2$} on 
$O$, if for each point $\bar{x} \in O$, there is a neighborhood around $\bar{x}$ and a scalar $\rho$ such that on this neighborhood $f+\rho|\cdot|^2$ is a finite convex function.
}
\end{definition}

\noindent
By Theorem \ref{ex}, the regular and limiting subdifferentials coincide for lower-$\bf{C}^2$ functions.

\begin{remark}
{\rm To illustrate the abundance of lower-$\bf{C}^2$ functions, consider the following example. Given $\bf{C}^2$ functions $f_i\colon O\to\R$ on an open set $O\subset\R^n$ ($i=1,\ldots,m$), the function $f=\max\{f_1,\ldots,f_m\}$ is lower-$\bf{C}^2$ on $O$. For more details see \cite[Chapter 10.F]{VA}.
}
\end{remark}

\subsection{Hausdorff Measures}

For a set $U\subset\R^n$, let $\mbox{\rm diam}\, U$ denote its diameter, that is $$\mbox{\rm diam}(U)=\sup_{x,y\in U} |x-y|.$$
\begin{definition}
{\rm Consider a set $S\subset\R^n$ and real numbers $\delta,d > 0$. We define $$\lambda_d^{\delta}(S)=\inf\Big\{\sum_{i=1}^{\infty}\mbox{\rm diam}(U_i)^d: S\subset\bigcup_{i=1}^{\infty}U_i,\, \mbox{\rm diam}(U_i)<\delta\Big\}.$$ 
}
\end{definition}
Observe the infimum in the definition above is taken over all countable covers $\{U_i\}$ of $S$, such that $\mbox{\rm diam}(U_i)<\delta$ for each $i$.

\begin{definition}
{\rm
For a set $S\subset\R^n$, define the {\em $d$-dimensional Hausdorff measure} of $S$ to be $$\lambda_d(S)=\lim_{\delta\rightarrow 0} \lambda_d^{\delta}(S).$$}
\end{definition}

It can be shown that for each $d>0$, the set function $\lambda_d$ is an outer measure on $\R^n$. Furthermore, if $d$ is a positive integer, then on Lebesgue measurable sets in $\R^d$ the $d$-dimensional Hausdorff measure is a rescaling of the $d$-dimensional Lebesgue measure. For more details, see \cite{hausdorff}. The following is an easy consequence of the definition of Hausdorff measure.

\begin{proposition}\label{prop:lip}  
Consider a set $S\subset\R^n$ and let $f\colon S\to\R^m$ be a Lipschitz continuous mapping with Lipschitz constant $\kappa$. Then for any real number $d>0$, we have $\lambda_d(f(S))\leq \kappa^d \lambda_d(S)$. 
\end{proposition}

\begin{corollary}\label{cor:local_lip}
Consider a set $S\subset\R^n$ and let $f\colon S\to\R^m$ be a locally Lipschitz mapping. Then for any real number $d>0$, if $\lambda_d(S)=0$ then $\lambda_d(f(S))=0$. 
\end{corollary}
\begin{proof}
Around each point $x\in S$, consider a neighborhood on which $f$ is Lipschitz continuous. This collection of neighborhoods forms a cover of $S$, and hence there is a countable subcover, say $\{V_i\}$. By Proposition~\ref{prop:lip}, for each index $i$ we have $\lambda_d(f(V_i))=0$, and hence $$\lambda_d(f(S))=\lambda_d(\bigcup_{i=1}^{\infty}f(V_i))\leq\linf_{n\rightarrow\infty} \displaystyle\sum\limits_{i=1}^{n} \lambda_d(V_i)=0,$$  as claimed.   
\end{proof}

\begin{definition}
Consider a compact, convex set $F\subset\R^n$. The set of maximizers $\argmax_{x\in F} \langle c,x\rangle$ is called the {\em exposed face} of the set
$F$ corresponding to the vector $c$. In particular, the set $F$ is itself an exposed
face (corresponding to $c = 0$). All other exposed faces are said to be {\em proper}.
\end{definition}

For a convex set $S\subset\R^n$, we will denote its closure, relative interior, and relative boundary by $\cl S$, $\ri S$, and $\rb S$, respectively. To prove the main result, we will need the following two theorems.
\begin{theorem}(Larman)\cite{Larman}\label{thm:larm} Let $S\subset\R^n$ be a compact convex set. Let $N$ be the union of the relative boundaries of all the proper exposed faces. Then $\lambda_{n-1}(N)=0$. 
\end{theorem}

\begin{theorem}\cite[Proposition 3]{gen}\label{thm:lewis}
Suppose zero lies in the interior of the compact convex set $F\subset\R^n$. Then the proper exposed faces of the polar set $F^{\circ}$ are those sets of the form $$G=\{c\in N_F(x):\langle c,x\rangle=1\},$$ for points $x$ on the boundary of F. Furthermore, any such exposed face has relative interior given by $$\ri G=\{c\in \ri N_F(x):\langle c,x\rangle= 1\}.$$
\end{theorem}

\section{Main Result}
\subsection{Subdifferentials of Convex Functions}
The unit sphere in $\R^n$ will be denoted by $\mathbb{S}^{n-1}$, and an open ball of radius $r$ around a point $x\in\R^n$ will be denoted by $B(x,r)$.
\begin{lemma}\label{lem:normal}
Let $F\subset\R^n$ be a convex set. Then 
$$\lambda_{n-1}\Big((\bigcup_{x\in F}\rb N_F(x))\cap\mathbb{S}^{n-1}\Big)=0.$$
\end{lemma}
\begin{proof} 
Observe that $N_F(x)=N_{\mbox{\scriptsize {\rm cl}}\, F}(x)$ for $x\in F$, so it is sufficient to show that the statement of the lemma holds for a closed convex set $F$. First, let us consider the case when $F$ is a compact convex set. Without loss of generality, we can assume that zero is in the interior of $F$, since otherwise we can translate $F$, so as to have $0\in \ri F$, and then consider $\R^n$ as the direct sum of the span of $F$ and its orthogonal complement. Define $$G:=\bigcup_{x\in F}\{c\in \rb N_F(x):\langle c,x\rangle= 1\}.$$ Combining Theorems~\ref{thm:larm} and~\ref{thm:lewis} , we deduce $\lambda_{n-1}(G) = 0$. Observe that $G$ is contained in $\R^n\setminus \{0\}$.
Now consider the mapping $$f:\R^n\setminus \{0\}\rightarrow\mathbb{S}^{n-1},$$
$$x\mapsto |x|^{-1}x.$$ The mapping $f$ is locally Lipschitz. Consequently, by Corollary~\ref{cor:local_lip}, we have $\lambda_{n-1}(f(G))=0$. Observe that the image set $f(G)$ is contained in $(\bigcup_{x\in F}\rb N_F(x))\cap\mathbb{S}^{n-1}$, since $f$ simply scales each element of $G$. Now, to see the reverse inclusion, consider a vector $c\in (\rb N_F(\bar{x}))\cap\mathbb{S}^{n-1}$ for some vector $\bar{x}\in F$. By definition of the normal cone, we have $$\langle c, \bar{x}-x\rangle \geq 0, ~~\text{{\rm for all }} x\in F.$$ In particular, since $0$ lies in the interior of $F$, we have $\langle c,\bar{x}\rangle > 0$. So we deduce $\widehat{c}:= |\langle c,\bar{x}\rangle|^{-1}c\in G$ and $f(\widehat{c})=c$. Thus we have shown $$f(G)= (\bigcup_{x\in F}\rb N_F(x))\cap\mathbb{S}^{n-1},$$ and consequently $$\lambda_{n-1}\Big((\bigcup_{x\in F}\rb N_F(x))\cap\mathbb{S}^{n-1}\Big)=0,$$ as we claimed. 

To get rid of the boundedness assumption on $F$, we will use a standard limiting argument. Assume that $F$ is a closed convex set that is not necessarily bounded. For a positive integer $k$, let $F_k = F\cap B(0,k)$. Observe 
\begin{align*} F_k&\uparrow F,\\ 
\Big(\bigcup_{x\in B(0,k)\cap F}\rb N_F(x)\Big)&\uparrow \Big(\bigcup_{x\in F}\rb N_F(x)\Big). 
\end{align*}
Thus we have 
\begin{align*}
\lambda_{n-1}\Big((\bigcup_{x\in F}\rb N_F(x))\cap\mathbb{S}^{n-1}\Big)&=\lim_{k\rightarrow\infty} \lambda_{n-1}\Big((\bigcup_{x\in B(0,k)\cap F}\rb N_F(x))\cap\mathbb{S}^{n-1}\Big)\\
&=\lim_{k\rightarrow\infty} \lambda_{n-1}\Big((\bigcup_{x\in B(0,k)\cap F}\rb N_{\overline{B(0,k)}\cap F}(x))\cap\mathbb{S}^{n-1}\Big)\\
&\leq\lim_{k\rightarrow\infty} \lambda_{n-1}\Big((\bigcup_{x\in \overline{B(0,k)}\cap F}\rb N_{\overline{B(0,k)}\cap F}(x))\cap\mathbb{S}^{n-1}\Big)\\
&= 0, 
\end{align*}
where the final equality follows since $\overline{B(0,k)}\cap F$ is a compact convex set.
\end{proof}

We need the following simple proposition.  For future reference, we let $\pi\colon\R^{n+1}\to\R^n$ be the canonical projection onto the first $n$ coordinates.

\begin{proposition}\label{prop:cone}
Consider a convex function $f\colon\R^n\to\overline{\R}$ and a point $x\in\R^n$. Then we have the relation,
$$v\in\rb\partial f(x) \Leftrightarrow (v,-1)\in \rb N_{\mbox{{\scriptsize {\rm epi}}}\, f}(x,f(x)).$$
\end{proposition}
\begin{proof}
 Let $K$ denote the normal cone, $N_{\mbox{{\scriptsize {\rm epi}}}\, f}(x,f(x))$. If $\partial f(x)=\emptyset$, then there is no $v\in\R^n$ such that $(v,-1)\in \rb K$, and hence the result holds trivially. Assume that $\partial f(x)$ is nonempty. Observe $$\mbox{\rm ri}\, K \not\subset \{y\in\R^{n+1}: y_{n+1}\geq 0\},$$ since otherwise taking closures gives $y_{n+1}\geq 0$ for all $y\in K$ and hence we have $\partial f(x)=\emptyset$, which is a contradiction. Thus there exists a point $y\in \mbox{\rm ri}\, K$ with $y_{n+1}<0$. Since $K$ is a cone, we can rescale to get $\hat{y}\in \ri K$ with $\hat{y}_{n+1}=-1$. Hence $$\ri K \cap \{y\in\R^{n+1}:y_{k+1}=-1\}\neq \emptyset.$$ Using \cite[Proposition 2.42]{VA}, we deduce that 
\begin{equation}\label{eq:ri}
\ri(K \cap \{y\in\R^{n+1}:y_{k+1}=-1\})= \ri K \cap \{y\in\R^{n+1}:y_{k+1}=-1\}.
\end{equation} 
Finally, we have $$\ri \partial f(x)= \pi\Big(\ri (K\cap \{y\in\R^{n+1}:y_{k+1}=-1\})\Big)=\{v:(v,-1)\in\ri K\},$$ where the last equality follows from (\ref{eq:ri}). Taking compliments, the result follows.
\end{proof}

\begin{theorem}\label{thm:main}
Let $f:\R^n\rightarrow\overline{\R}$ be a convex function. Then the set $$\bigcup_{x\in\R^n} \rb \partial f(x)$$ is Lebesgue null. 
\end{theorem}
\begin{proof}
Let $$H_{-1}:=\{x\in\R^{n+1}:x_{n+1}=-1\},$$ $$H_{<}:=\{x\in\R^{n+1}:x_{n+1}<0\},$$ $$K:=\Big(\bigcup_{x\in \mbox{{\scriptsize {\rm dom}}}\, f} \rb N_{\mbox{{\scriptsize {\rm epi}}}\, f}(x,f(x))\Big)\cap\mathbb{S}^n\cap H_{<}.$$ Applying Lemma~\ref{lem:normal} to $\epi f$, we deduce  
$\lambda_n(K)=0$. Consider the mapping $$\phi: H_{<}\rightarrow H_{-1},~~c\mapsto |c_{n+1}|^{-1}c.$$ Observe that $\phi$ is locally Lipschitz, and therefore by Corollary~\ref{cor:local_lip}, we have $\lambda_n(\phi(K))=0$.
 From Proposition~\ref{prop:cone}, we have $$\pi\circ\phi(K)=\bigcup_{x\in \mbox{{\scriptsize {\rm dom}}}\, f} \rb \partial f(x).$$  
Since $\pi$ is Lipschitz as well, we deduce $\lambda_n(\bigcup_{x\in\R^n} \rb \partial f(x))=0$. Since Hausdorff measures are Borel-regular \cite[Section 2.10.2]{geo_mes}, the set $\bigcup_{x\in\R^n} \rb \partial f(x)$ is Lebesgue measurable and has Lebesgue measure zero. 
\end{proof}
\begin{definition}
{\rm Consider a convex function $f\colon\R^n\rightarrow\overline{\R}$. A minimizer $x\in\R^n$ of $f$ is said to be {\em nondegenerate} if it satisfies the property $0 \in \ri \partial f(x)$.}
\end{definition}

\begin{corollary}
Let $f\colon\R^n\to\overline{\R}$ be a proper convex function. Consider the collection of perturbed functions $f_v(x)=f(x)-\langle v,x\rangle$, indexed by vectors $v\in\R^n$. Then for a full measure set of vectors $v\in\R^n$, the function $f_v$ has at most one minimizer, which furthermore is nondegenerate.
\end{corollary}
\begin{proof} 
The uniqueness part of the claim is classical, as discussed in the introduction. Thus it is sufficient to show that for a full measure set of vectors $v\in\R^n$, every critical point of $f_v$ is nondegenerate. Indeed, we have  $0\in\rb\partial f_v(x) \Leftrightarrow v\in\rb\partial f(x)$. By Theorem~\ref{thm:main}, the set of vectors $v$ for which $v\in\rb\partial f(x)$ for some $x\in\R^n$ has Lebesgue measure zero, and so the result follows.
\end{proof}

\subsection{Extension to lower-$\bf{C}^2$ functions}
Having proved Theorem~\ref{thm:main}, we can now easily extend this theorem to a nonconvex situation.  In particular, shortly we will show that an analogous statement holds for all lower-$\bf{C}^2$ functions.
\begin{theorem}\label{thm:ext}
Consider a proper function $f\colon\R^n\to\overline{\R}$ with the property that for any point $\bar{x}$ in its domain, there is a neighborhood $V$ around $\bar{x}$ such that on $V$, the function $f$ admits the representation $f=g-\frac{1}{2}\rho|\cdot|^2$, where $g$ is a convex function and $\rho$ is a positive real number. Then the set  $$\bigcup_{x\in \R^n} \rb \partial f(x)$$ is Lebesgue null.  
\end{theorem}
\begin{remark} In Theorem~\ref{thm:ext}, unlike in the definition of lower-$\bf{C}^2$ functions, the domain of $f$ is not required to be an open set and the convex function $g$ in the local representation of $f$ is not required to be finite.
\end{remark}
\begin{proof}
For each point $x\in\dom f$, consider the neighborhood guaranteed to exist by our assumption on $f$. This collection of neighborhoods is an open cover of the domain of $f$, and hence has a countable subcover, say $\{V_i\}$. Consider an arbitrary set $V_i$ from this cover. On $V_i$, we have $f=g-\frac{1}{2}\rho|\cdot|^2$, and hence 
\begin{align}\label{eqn:equiv} 
\bigcup_{x\in V_i} \rb \partial f(x)&=\bigcup_{x\in V_i\cap\mbox{{\scriptsize {\rm dom}}}\, f} \rb \partial g(x)-\rho x\\
&=\bigcup_{x\in V_i\cap\mbox{{\scriptsize {\rm dom}}}\, f} \rb (\partial g(x)+x)-(\rho+1)x.\nonumber  
\end{align}
Consider the map $$H\colon \bigcup_{x\in V_i\cap\mbox{{\scriptsize {\rm dom}}}\, f} \rb (\partial g(x)+x) \to \bigcup_{x\in V_i} \rb \partial f(x),$$
$$c\mapsto c-(\rho+1)(\partial g + I)^{-1}(c).$$
In light of (\ref{eqn:equiv}) and Theorem~\ref{thm:minty}, the mapping $H$ is well-defined, surjective, and Lipschitz continuous. Observe 
\begin{equation*}\label{eq:lip}
\lambda_n\Big(\bigcup_{x\in V_i\cap\mbox{{\scriptsize {\rm dom}}}\, f} \rb (\partial g(x)+x)\Big)=\lambda_n\Big(\bigcup_{x\in V_i\cap\mbox{{\scriptsize {\rm dom}}}\, f} \rb \partial (g(\cdot)+\frac{1}{2}|\cdot|^2)(x)\Big)=0,
\end{equation*}
where the last equality follows from convexity of $g+\frac{1}{2}|\cdot|^2$ and Theorem~\ref{thm:main}. From the equation above and Corollary~\ref{cor:local_lip}, we have $\lambda_n\Big(\bigcup_{x\in V_i} \rb \partial f(x)\Big)=0$.
Since Hausdorff measures are Borel-regular, the set $\bigcup_{x\in V_i} \rb \partial f(x)$ is Lebesgue measurable and has Lebesgue measure zero. Finally, since $\{V_i\}$ is a countable cover of $\dom f$, it easily follows from a limiting argument that $\bigcup_{x\in \R^n} \rb \partial f(x)$ is a Lebesgue null set, as was claimed.
\end{proof}

\begin{corollary}\label{cor:lower}
Let $f\colon O\rightarrow\overline{\R}$ be a lower-$\bf{C}^2$ function on an open set $O\subset\R^n$. Then the set $$\bigcup_{x\in\R^n} \rb \partial f(x).$$ is Lebesgue null. 
\end{corollary}
\begin{proof}
From Definition~\ref{thm:char}, $f$ satisfies the conditions of Theorem~\ref{thm:ext}, and hence the result follows. 
\end{proof}

\begin{definition}
{\rm Let $f\colon O\to\R$ be a lower-$\bf{C}^2$ function on an open set $O\subset\R^n$. We say that a point $x\in\R^n$ is {\em critical} for the function $f$ if $0\in\partial f(x)$, and we call such a critical point $x$ {\em nondegenerate} if the stronger property $0 \in \mbox{ri}\, \partial f(x)$ holds.}
\end{definition}

\begin{corollary}
Let $f\colon O\to\R$ be a lower-$\bf{C}^2$ function on an open set $O\subset\R^n$. Consider the collection of perturbed functions $f_v(x)=f(x)-\langle v,x\rangle$, indexed by vectors $v\in\R^n$. Then for a full measure set of vectors $v\in\R^n$, every critical point of the function $f_v$ is nondegenerate.
\end{corollary}
\begin{proof} 
We have  $0\in\rb\partial f_v(x) \Leftrightarrow v\in\rb\partial f(x)$. By Corollary~\ref{cor:lower}, the set of vectors $v$ for which $v\in\rb\partial f(x)$ for some $x\in\R^n$ has Lebesgue measure zero, and so the result follows.
\end{proof}

\section{A conjecture}
We can formulate Theorem~\ref{thm:main} in terms of monotone set-valued mappings. See \cite[Chapter 12]{VA} for the definitions. If we restrict our attention in the theorem to closed proper convex functions $f$, then Theorem~\ref{thm:main} is equivalent to the statement that for a maximal cyclically-monotone mapping $F\colon\R^n\rightrightarrows\R^n$, the image of the set-valued map $x\mapsto \rb F(x)$ has Lebesgue measure zero (see \cite[Theorem 12.25]{VA}). We make the following related conjecture. 
\begin{conjecture}\label{conj:max}
Let $F\colon\R^n\rightrightarrows\R^n$ be a maximal monotone mapping. Then the image of the map $x\mapsto \rb F(x)$ has Lebesgue measure zero, that is, the set  $$\bigcup_{x\in\R^n} \rb F(x)$$ is Lebesgue null.
\end{conjecture}
A proof of Conjecture~\ref{conj:max}, along with the techniques presented in this paper, might extend the result of Corollary~\ref{cor:lower} to the class of ``prox-regular'' functions \cite{prox}.

\bibliographystyle{amsplain}
\bibliography{Larman}

\end{document}